\documentclass{amsart}

\usepackage[T1]{fontenc} 
\usepackage{amsmath}
\usepackage{amsthm} 
\usepackage{amscd} 
\usepackage{amssymb}
\usepackage{latexsym} 

\usepackage[english]{babel}

\usepackage{fullpage}

\def\epsilon{\varepsilon}
 
\newcommand{\ssm}{\smallsetminus}

\renewcommand{\phi}{\varphi}

\newcommand{\Curr}{\mathbb{P}\text{Curr}}
\newcommand{\curr}{\text{Curr}}
\newcommand{\Out}{\text{Out}}

\newcommand{\inv}{^{-1}}

\newcommand{\dom}{\text{dom}}

\newcommand{\Teich}{\text{Teich}}
\newcommand{\MCG}{\text{MCG}}

\newcommand{\FN}{\mathbb{F}_N} % F ou F_n ou F_N ?
\newcommand{\cvn}{\text{cv}_N}
\newcommand{\CVN}{\text{CV}_N}
\newcommand{\barCVN}{\bar{\text{CV}}_N} 
\newcommand{\barcvn}{\bar{\text{cv}}_N}
\newcommand{\CQ}{{\mathcal Q}}
\newcommand{\Tobs}{\widehat T^{\scriptstyle\text{obs}}}

\newcommand{\GE}{\text{GE}}

\newcommand{\R}{\mathbb R} 
\newcommand{\Z}{\mathbb Z}
 
\newcommand{\N}{\mathbb N}

\def\bar{\overline} 
\def\tilde{\widetilde} 
\def\hat{\widehat}
\newtheorem{thm}{Theorem}[section]
 
\newtheorem{lem}[thm]{Lemma}
\newtheorem{prop}[thm]{Proposition}

\newtheorem*{thm*}{Theorem}
\newtheorem*{prop*}{Proposition}
\newtheorem*{thm-main*}{Theorem~\ref{thm:main}}

\newtheorem*{defn*}{Definition} 

\newtheorem*{rem*}{Remark}

\numberwithin{equation}{section} 
\begin{document}

\title{Ergodic currents dual to a real tree}
%\\
%\rule{5cm}{1pt}\\
%Courants ergodiques duaux à un arbre réel}

\author{Thierry Coulbois, Arnaud Hilion} \thanks{The authors are
  supported by the grant ANR-10-JCJC 01010 of the Agence nationale de
  la recherche.}

\date{\today }

%% \selectlanguage{francais}
%% \begin{abstract}Soit $T$ un un arbre réel au bord de l'outre
%%   espace. Quand l'action du groupe libre $\FN$ sur $T$ est libre et
%%   avec des orbites denses, nous démontrons qu'il y a au plus $3N-5$
%%   classes projectives de courants ergodiques duaux à $T$.

%%   Nous définissons une induction de dépliage en combinant l'induction
%%   de Rips et les découpages pour de tels arbres. Pour un courant $\mu$
%%   dual à $T$, l'induction de dépliage construit une suite
%%   d'approximations qui converge vers $\mu$.

%%   Nous donnons aussi un critère d'unique ergodicité.

%% \medskip

%% \begin{center}
%% \rule{5cm}{1pt}
%% \end{center}

%% \medskip

%% \selectlanguage{english}

%% \noindent\textsc{Abstract.}  

\begin{abstract} 
Let $T$ be an $\R$-tree with dense orbits in the boundary of Outer
  space. When the free group $\FN$ acts freely on $T$, we prove that
  the number of projective classes of ergodic currents dual to $T$ is
  bounded above by $3N-5$.

We combine Rips induction and splitting induction to define unfolding
induction for such an $\R$-tree $T$. Given a current $\mu$ dual to
$T$, the unfolding induction produces a sequence of approximations
converging towards $\mu$.

We also give a unique ergodicity criterion.
\end{abstract}

%\selectlanguage{english}

\maketitle

%\tableofcontents

\section{Introduction}

\subsection{Main results}

Let $\FN$ be the free group with $N$ generators. M.~Culler and
K.~Vogtmann~\cite{cv-moduli} introduced \textbf{Outer space}: the
space $\CVN$ of projective classes of free minimal actions of $\FN$ by
isometries on simplicial metric trees. We denote by $\cvn$ the space
of free minimal actions of $\FN$ by isometries on simplicial metric
trees: $\cvn$ is the unprojectivised Outer space.
The space $\CVN$ is a contractible simplicial complex,
with missing faces. The maximal dimension of a simplex in $\CVN$ is
$3N-4$.  Outer space admits a Thurston boundary, denoted by
$\partial\CVN$, which gives rise to a compactification
$\barCVN=\CVN\uplus\partial\CVN$ of $\CVN$.  This compactification
consists of projective classes of $\R$-trees with a minimal, very
small action of $\FN$ by isometries
\cite{cl-verysmall,bf-stable}.  Again, we denote by
$\partial\cvn$ and $\barcvn$ the corresponding unprojectivized spaces.

The group $\Out(\FN)$ of outer automorphisms of the free group $\FN$
acts on $\CVN$.  One basically considers that the Outer space
plays the same role for $\Out(\FN)$ as the Teichm\"uller space of a
surface $S$ for the mapping class group $\MCG(S)$ of $S$ -- see for
instance \cite{bv-out-teich} where the analogy is carried on.
 
Another space on which $\Out(\FN)$ naturally acts, and which appears
to capture slightly different informations from $\Out(\FN)$, is the
space $\Curr(\FN)$ of projective classes of
currents~\cite{kapo-currents,kapo-frequency}.  Let us recall what a
current is.

The space $\partial\FN$ of ends of the free group $\FN$ is a Cantor
set, equipped with an action by homeomorphisms of $\FN$. The space
$\partial\FN$ is also the Gromov boundary of $\FN$. The action of
$\FN$ on itself by left multiplication extends continuously to
$\partial\FN$.  We denote by $\partial^2\FN=(\partial\FN)^2\ssm\Delta$
the double boundary of $\FN$, where $\Delta$ stands for the
diagonal. The double boundary inherits a product topology from
$\partial\FN$, and the action of $\FN$ on $\partial\FN$ gives rise to
a diagonal action on $\partial^2\FN$.  The involution
$(X,Y)\mapsto(Y,X)$ of $\partial^2\FN$ is called the flip map.  A
\textbf{current} $\mu$ is a $\FN$-invariant, flip-invariant, Radon
measure (that is to say a Borel measure which is finite on compact
sets) on $\partial^2\FN$.  We notice that a linear combination, with
non-negative coefficients, of currents is still a current.

The space $\curr(\FN)$ of currents of $\FN$ is equipped with the
weak-$*$-topology: it is a locally compact space.  The group
$\Out(\FN)$ naturally acts on $\curr(\FN)$ -- see for instance
\cite{kapo-currents} for a detailed description of this action. The
space $\Curr(\FN)$ of projective classes of (non-zero) currents
equipped with the quotient topology is a compact set, and the action
of $\Out(\FN)$ on $\curr(\FN)$ induces an action on $\Curr(\FN)$.

The spaces $\curr(\FN)$ and $\cvn$ can be viewed, to some extent, 
as dual spaces. Indeed, M.~Lustig and I.~Kapovich~\cite{kl-geom-intersection}
defined a kind of duality bracket:
\[
\begin{array}{cccc}
\langle\cdot,\cdot\rangle: & \barcvn \times \curr(\FN) & \rightarrow & \R_{\geq 0}\\
& (T,\mu) & \mapsto & \langle T,\mu \rangle.
\end{array}
\]
This bracket is characterized as being the unique continuous $\Out(\FN)$-equivariant
map $\barcvn \times \curr(\FN) \rightarrow \R_{\geq 0}$
which is $\R_{\geq0}$-homogeneous in the first entry and 
$\R_{\geq0}$-linear in the second one, 
and which satisfies that for all $T\in\barcvn$ and for all rational
current $\mu_w$ induced by a conjugacy class of a primitive element $w$ of $\FN$,
$\langle T,\mu_w \rangle$ is the translation length of the conjugacy class of
$w$ in $T$ -- see \cite{kl-geom-intersection}.
The number $\langle T,\mu \rangle$ is the \textbf{intersection number} of
$T$ and $\mu$.

A tree $T\in\barcvn$ and a current $\mu\in\curr(\FN)$ are {\bf dual}
when $\langle T,\mu \rangle =0$.  We notice that the nullity of the
intersection number $\langle T,\mu \rangle =0$ only depends on the
projective classes of $T$ and $\mu$.  The main goal of this paper is
to explain how to build all the currents dual to a given tree.  We
remark that the set of projective classes of currents dual to a given
tree $T$ is convex: the extremal points of this set are {\bf ergodic}
currents.

\begin{thm}\label{thm:main}
Let $T$ be an $\R$-tree with a free, minimal action of $\FN$ by
isometries with dense orbits. 
There are at most $3N-5$ projective classes of ergodic currents
dual to $T$. 
\end{thm}

The duality between trees and currents can also be unterstood by
considering laminations.  An $\R$-tree $T$ with an action of $\FN$ by
isometries with dense orbits has a dual lamination
$L(T)$ \cite{chl1-II}. I.~Kapovich and
M.~Lustig~\cite{kl-intersection} proved that a current is dual to $T$
if and only if it is carried by the dual lamination. Thus we can
rephrase our main Theorem to:

\begin{thm}
Let $T$ be an $\R$-tree with a free, minimal action of $\FN$ by
isometries with dense orbits. The dual lamination $L(T)$ carries at
most $3N-5$ projective classes of ergodic currents. 
\end{thm}

%%%%%%%%%%%%%%%%%%%%%%%%%%%%%%%%%%%%%%%%%%%%%%%%%%%%%%%%%%%%%%%%%%%%%%%
%%%%%%%%%%%%%%%%%%%%%%%%%%%%%%%%%%%%%%%%%%%%%%%%%%%%%%%%%%%%%%%%%%%%%%%

\subsection{The surface case} 

Generally speaking, the present work is inspired by the situation of
hyperbolic surfaces which we recall here.

Let $S=S_g$ be an oriented surface of genus $g$ with negative Euler
characterisic $\chi(S)=2-2g<0$. The mapping class group of $S$ acts on
Teichm\"uller space $\Teich(S)$ which is homeomorphic to a ball of
dimension $6g-6$. Teichmüller space can be compactified by adding its
Thurston boundary $\partial\Teich(S)$ which is homeomorphic to a
sphere of dimension $6g-7$. There are several useful models to
describe $\partial\Teich(S)$, see for instance
\cite{paulin-top-equiv}. In particular, points in $\partial\Teich(S)$
can be seen alternatively as projective classes of:
\begin{itemize}
\item $\R$-trees with a minimal small action by isometries of
  the fundamental group of the surface,
\item measured geodesic laminations on $S$,
\item measured singular foliations on $S$ (up to Whitehead equivalence).
\item currents (see~\cite{bona-teich-curr})
\end{itemize}
%Indeed, such a measured geodesic lamination (or measured sigular
%foliation) defines a transverse $\R$-tree.

A geodesic lamination may carry more than one transverse measure (even
if the lamination is minimal). The simplex of measures carried by a
given geodesic lamination embeds in the boundary of Teichm\"uller
space and thus has dimension less than $6g-5$.  In fact,
G.~Levitt~\cite{levitt-these} proved that an orientable foliation
carries at most $3g-3$ distinct projective ergodic measures, see also
\cite{papa-86}.

It is not easy to exhibit a non-uniquely ergodic minimal
foliation. Examples come from interval exchange transformations
(IET). Indeed, the mapping torus of an IET gives rise to a foliated
surface (through zippering process). H.~Keynes and D.~Newton~\cite{kn}
and M.~Keane \cite{keane-non-ue-iet} gave examples of non-uniquely
ergodic minimal IET (and thus of non-uniquely ergodic minimal
foliations on a surface). E.~Sataev~\cite{sataev} constructed minimal
IET with exactly $k$ ergodic measures for $1\leq k\leq 2g$. The main
tool to investigate ergodic properties of the foliation of an IET is
to use a suitable induction such as Rauzy-Veech induction.

In the case of Outer space, there are fewer models available to
describe points in the boundary $\partial\CVN$. In particular,
I.~Kapovich and M.~Lustig~\cite{kl-incompatibility} proved that there
is no $\Out(\FN)$-equivariant continuous embedding of $\partial\CVN$
in the space of currents.  Although D.~Gaboriau and
G.~Levitt~\cite{gl-rank} proved that the dimension of $\partial\CVN$
is $3N-5$ (see also V.~Guirardel~\cite{guir-dynamic} where this result is
reinterpreted in terms of length measures), we cannot deduce from this
fact a bound on the number of ergodic currents dual to an $\R$-tree.
This leads us to follow the alternative strategy of analysing the
ergodic properties via an induction.

%%%%%%%%%%%%%%%%%%%%%%%%%%%%%%%%%%%%%%%%
%%%%%%%%%%%%%%%%%%%%%%%%%%%%%%%%%%%%%%%%

\subsection{Outline of the proof}

The general strategy to prove Theorem~\ref{thm:main} consists in
mimicking what has been done for interval exchange
transformations. 
However, a tree in $\partial\cvn$ is not necessarily transverse to the
foliation of an interval exchange transformation. Thus we cannot use
the Rauzy-Veech induction.  Instead, for an $\R$-tree $T$ in
$\partial\CVN$ with dense orbits and a basis $A$ of $\FN$, we defined,
together with M.~Lustig~\cite{chl4}, a kind of a band complex
$S_A=(K_A,A)$: the mapping torus of the system of isometries on the
compact heart given by $A$. For such a band complex we can use the unfolding
induction which consists in either the Rips induction~\cite{ch-a} or
the splitting induction~\cite{chr}.

First, let us recall that a current can be seen as a Kolmogorov
function. Given a graph $\Gamma$ and a marking isomorphism
$\FN\stackrel{\sim}{\to}\pi_1(\Gamma)$, the universal cover $\tilde\Gamma$ is a
(simplicial) tree with an action of $\FN$. Points in $\partial\FN$
correspond to points in the boundary of $\partial\tilde\Gamma$ and
elements $(X,Y)\in\partial^2\FN$ are represented by bi-infinite lines in
$\tilde\Gamma$. The topology on $\partial^2\FN$ is given by the
cylinders: sets of lines that share a common subpath. Currents are
completely described by the measures of the cylinders: For a current
$\mu$ and a finite path $\gamma$ in $\Gamma$ (or rather its lift in
$\tilde\Gamma$), the Komolgorov function associated to $\mu$, assigns
to $\gamma$ the measure $\mu(\gamma)$ of the cylinder defined by
$\gamma$. (Note that the fuzzyness about the absence of base-point in
the marking isomorphism or the choice of a lift of $\gamma$ in
$\tilde\Gamma$ is justified by the $\FN$-invariance of currents.)

For a marked graph $\Gamma$ and a current $\mu$ we consider the
non-negative vector $\mu_\Gamma=(\mu(e))_{e\in E(\Gamma)}$ of the
$\mu$-measure of the cylinders defined by the edges of $\Gamma$. The
unfolding induction, starting from an indecomposable $\R$-tree $T$ in
$\partial\CVN$, produces a sequence $(\Gamma_n)$ of marked graphs and
a sequence of non-negative integer matrices $(M_n)$ such that for any
current $\mu$ dual to $T$:
\begin{enumerate}
\item\label{item:split-ind-1} the sequence of vectors
  $(\mu_{\Gamma_n})$ completely determines $\mu$,
\item\label{item:split-ind-2} $\mu_{\Gamma_n}=M_n\mu_{\Gamma_{n+1}}$.
\end{enumerate}
Using properties~(\ref{item:split-ind-1}) and
(\ref{item:split-ind-2}), we derive that the bound on the the
dimension of the simplex of currents dual to $T$ is given by the Euler
characteristic of the graphs $\Gamma_n$ which remains constant through
splitting induction.

In fact, the sequence of marked graphs $(\Gamma_n)$ given by the
unfolding induction limits to the support of the currents dual to
$T$. It also gives a decomposition of the cylinders of the dual
lamination $L(T)$ defined by the edges of $\Gamma_n$ similar to the
Kakutani-Rokhlin towers approximations of a dynamical system -- see
for instance~\cite[Definition~6.4.1]{durand-cant6} .

In Section~\ref{sec:ue}, we adress the question of unique
ergodicity: an $\R$-tree in $\partial\CVN$ is uniquely ergodic if it
is dual to a unique projective current. In the context of IET, there is
a famous sufficient condition known as Masur
criterion~\cite{masur}. This criterion can be understood using the
sequence of matrices of the Rauzy-Veech induction, see for instance
\cite{yocc-2005}.  We derive such a criterion for $\R$-trees in
$\partial\CVN$ in terms of the sequence of matrices of the splitting
induction.

\subsection*{Acknowledgements}

We would like to thank John Smillie and Gilbert Levitt for giving us
key references and Hossein Namazi and Alexandra Pettet for pointing
out mistakes in previous versions.

\section{Preliminaries}

\subsection{Trees and dual laminations}

An $\R$-tree $T$ is a $0$-hyperbolic metric space. It has a Gromov
boundary $\partial T$.  We denote by $\hat T=\bar T\cup\partial T$ the
union of the metric completion of $T$ and its Gromov boundary. This is
a topological space which is not compact in general. A
\textbf{direction} $d$ at a point $P$ in $\hat T$ is a connected
component of $\hat T\ssm\{P\}$. We weaken the topology on $\hat T$ by
considering the set of directions as a sub-basis of open sets, we
denote by $\Tobs$ the resulting topological space which is Hausdorff,
compact and has exactly the same connected subsets as $\hat
T$~\cite{chl2}. Indeed $\Tobs$ is a dendrite in the terminology of
B.~Bowditch \cite{bowd-tree}.

We denote by $\partial\FN$ the Gromov-boundary of $\FN$: $\partial\FN$
is a Cantor set. Fixing a basis $A$ of $\FN$, elements of $\FN$ are
finite reduced words in $A^{\pm 1}$ and elements of $\partial\FN$ are
infinite reduced words in $A^{\pm 1}$. Let $T$ be an $\R$-tree in
$\partial\cvn$ with dense orbits. For a point $P$ in $T$ the orbit map
$\FN\to T$, $u\mapsto uP$, has a unique continuous extension to a map
$\CQ:\partial\FN\to\Tobs$. This map $\CQ$ does not depend on the
choice of $P$ \cite{ll-north-south, ll-periodic,chl2}.

We denote by $\partial^2\FN=(\partial\FN)^2\ssm\Delta$ the double
boundary of $\FN$, where $\Delta$ stands for the diagonal. The \textbf{dual
lamination} \cite{chl1-I,chl1-II} of the tree $T$ is
\[
L(T)=\{(X,Y)\in\partial^2\FN\ |\ \CQ(X)=\CQ(Y)\}.
\]
This is a closed, $\FN$-invariant, flip-invariant, subset of
$\partial^2\FN$.

The map $\CQ$ induces a continuous \cite{chl4} map $\CQ^2: L(T)\to\bar
T$ (here the topology on $\bar T$ is the metric topology). The
\textbf{limit set} $\Omega=\CQ^2(L(T))$ of $T$ is the image of the
lamination by $\CQ$. The tree $T$ is of \textbf{Levitt type} if the
limit set is totally disconnected.

\subsection{Compact heart}

Let $A$ be a basis of $\FN$. Elements of $\FN$ are vertices of the
Cayley graph of $\FN$ with respect to the basis $A$. Elements of
$\partial\FN$ are infinite reduced paths starting at $1$. Elements of
$\partial^2\FN$ are identified with bi-infinite reduced
paths in the Cayley graph indexed by $\Z$.  The \textbf{unit cylinder} of
$\partial^2\FN$ is the set of bi-infinite reduced paths
going through $1$ at index $0$:
\[
C_A(1)=\{(X,Y)\in\partial^2\FN\ |\ X_0\neq Y_0\}
\]
where $X_0$ is the first letter of the infinite word $X$.  The unit
cylinder $C_A(1)$ is a compact set. We denote by $L_A(T)=L(T)\cap
C_A(1)$ the symbolic lamination relative to $A$. It is a compact
set. Its image by $\CQ^2$ is the \textbf{compact limit set} relative to $A$:
\[
\Omega_A=\CQ^2(L_A(T))=\CQ^2(L(T)\cap C_A(1))\subseteq\bar T.
\]
The convex hull of $\Omega_A$ is the \textbf{compact heart} $K_A\subseteq\bar
T$ of the tree $T$ relative to $A$.

For each element $a\in A$ of the basis we consider its restriction to
$K_A$ as a partial isometry. By allowing inverses and composition we
get a pseudo-action of $\FN$ on $K_A$. We denote by $S_A=(K_A,A)$ this
system of isometries \cite{chl4}.

An element $u$ of $\FN$ is \textbf{admissible} if it is non-empty as a
partial isometry. An infinite reduced word $X\in\partial\FN$ is
\textbf{admissible} if all its prefixes are admissible. In this case
the domains of the prefixes form a nested sequence of non-empty compact
subtrees of the compact heart $K_A$ and the \textbf{domain} of $X$ is
their intersection, it is exactly the image of $X$ by the map $\CQ$:
\[
\{\CQ(X)\}=\dom X=\bigcap_{n\in\N}\dom(X_n)
\]
where $X_n$ is the prefix of length $n$.  A bi-infinite reduced word
$Z$ has two halves and we write $Z=(Z_-,Z_+)\in C_A(1)$. It is
\textbf{admissible} if all its finite factors are admissible, or
equivalently, if its two halves are admissible with the same domain.
In this case, the \textbf{domain} of $Z$ is the common domain of its two
halves:
\[
\dom(Z)=\{\CQ(Z)\}=\{\CQ(Z_+)\}=\{\CQ(Z_-)\}.
\]
The set of admissible bi-infinite words of the system of isometries
$S_A=(K_A,A)$ is the \textbf{admissible lamination} $L(S_A)$. This is
a shift-invariant symmetric closed subset of the full-shift of
bi-infinite reduced words in $A^{\pm 1}$. We remark that a bi-infinite
reduced word $Z$ corresponds to a pair $(Z_-,Z_+)$ in $\partial^2\FN$,
and for a pair $(X,Y)\in\partial^2\FN$, we can define the bi-infinite
reduced word $X\inv Y$. This correspondance is rather between a
shift-orbit of bi-infinite reduced words in $A^{\pm 1}$ and an
$\FN$-orbit in $\partial^2\FN$. By abuse of notations, using the above
correspondance, we have: 

\begin{prop}[\cite{chl4}]\label{prop:lam-repr}
  Let $T$ be an $\R$-tree with a minimal, very small action of $\FN$
  by isometries with dense orbits. Let $A$ be a basis of $\FN$ and let
  $S_A=(K_A,A)$ be the associated system of isometries.

  Then, the admissible lamination of $S_A$ is equal to the dual
  lamination of $T$: $L(S_A)=L(T)$.
\end{prop}

For a word $w\in\FN$, the \textbf{cylinder} $C_A(w)\subseteq L(S_A)$
is the set of bi-infinite admissible words $Z$ in $S_A$ which reads
$w$ at index $0$: $w$ is a prefix of the positive half $Z_+$.  The
shifts of the cylinders form a sub-basis of open sets of the
admissible lamination.

\subsection{Systems of isometries}

More generally a \textbf{system of isometries} $S=(F,A)$ is a compact
forest $F$ together with a finite set $A$ of partial isometries of
$F$. Here a \textbf{compact forest} is a disjoint union of finitely
many compact $\R$-trees. A partial isometry is a non-empty isometry
between two compact subtrees. The graph $\Gamma=\Gamma(S)$ of such a
system of isometries has the connected components of $F$ as vertices
and the partial isometries in $A$ as edges. The edge $a\in A$ goes
from the connected component that contains its domain to the connected
component of that contains its image. We always assume that $\Gamma$
is connected.

A finite reduced path $\gamma$ in $\Gamma$ is \textbf{admissible} if
it is non-empty as a partial isometry. Its domain $\dom(\gamma)$ is a
non-empty compact subtree of $F$. The infinite reduced path $X$ in
$\Gamma$ is \textbf{admissible} if all its subpaths are
admissible. Its domain $\dom(X)$ is the intersection of the nested
domains of the initial subpaths of $X$. A bi-infinite reduced
path $Z=\cdots Z_{-1}Z_0Z_1Z_2\cdots$ has to halves
$Z_-=Z_0\inv Z_{-1}\inv\cdots$ and $Z_+=Z_1Z_2\cdots$. It is
\textbf{admissible} if all its subpaths are admissible or equivalently
if its two halves are admissible and its \textbf{domain}
$\dom(Z)=\dom(Z_-)\cap\dom(Z_+)$ is a non-empty compact subtree of
$F$.

The \textbf{admissible lamination} $L(S)$ of the system of isometries
$S=(F,A)$ is the set of bi-infinite admissible paths in $\Gamma$. For
a finite reduced path $\gamma$ in $\Gamma$, the \textbf{cylinder}
$C(\Gamma,\gamma)$ is the set of bi-infinite admissible paths $Z$ that
goes through $\gamma$ at index $0$: $\gamma$ is a prefix of the
positive half $Z_+$.

Following D.~Gaboriau~\cite{gab-indgen}, a system of isometries has
\textbf{independent generators} if the domain $\dom(X)$ of any
infinite admissible path $X$ in $\Gamma$ contains a single point which
we denote by $\CQ(X)$.

\subsection{Currents}\label{sec:currents}

A \textbf{current} for the free group $\FN$ is an $\FN$-invariant,
flip-invariant, Radon measure (that is to say a Borel measure which is
finite on compact sets) on $\partial^2\FN$. The support of a current
is a lamination.

Let $T$ be an $\R$-tree with a minimal very small action of $\FN$ by
isometries with dense orbits. We denote by $\Curr(T)$ the simplex of
currents carried on the dual lamination $L(T)$. Let $A$ be a basis of
$\FN$ and $S_A=(K_A,A)$ be the system of isometries on the compact
heart of $T$. Recall that we identify the dual lamination $L(T)$ with
the admissible lamination $L(S_A)$ which is the shift of bi-infinite
admissible words in $A^{\pm 1}$. In this setting a current if a
shift-invariant, flip-invariant finite Borel measure on
$L(S_A)$~\cite{kapo-currents,chl1-III}.

For a current $\mu\in\Curr(T)$ and a word $w\in\FN$, we
denote by $\mu(w)=\mu(C_A(w))$ the finite measure of the cylinder of
$w$.  As the cylinders generate the topology, the current $\mu$ is
completly determined by the values $\mu(w)$, for $w\in\FN$.

\begin{prop}\label{prop:periodic-leaves}
  Let $T$ be an $\R$-tree with a free minimal action of $\FN$ by
  isometries with dense orbits. Let $\mu\in\Curr(T)$ be a current
  carried by the dual lamination of $T$.
  Then $\mu$ has no atomes.
\end{prop}
\begin{proof}
  Let $A$ be a basis of $\FN$ and let $S_A=(K_A,A)$ be the associated
  system of isometries.  Let $Z$ be a bi-infinite admissible word such
  that $\mu(Z)>0$. Then, as $\mu$ is shift-invariant, denoting by
  $\sigma$ the shift-map
\[
\mu(\{\sigma^nZ\ |\ n\in\Z\})=\mu(Z)\cdot\big|\{\sigma^nZ\ |\ n\in\Z\}\big|\leq\mu(L(S_A))<\infty.
\]
Thus the shift-orbit of $Z$ is finite and there exists $n>0$ such that
$\sigma^nZ=Z$. Let $u$ be the prefix of the positive half $Z_+$ of
length $n$. We have $u\inv Z_+=Z_+$ and using the equivariant map
$\CQ$
\[
\CQ(Z_+)=u\inv\CQ(Z_+).
\]
But the action of $\FN$ on $T$ is free, a contradiction.
\end{proof}

\subsection{Non-negative matrices}\label{sec:matrices}

We recall here basic facts that are folklore in ergodic theory. The
statements here are suitable in our context. In the analog context of
interval exchange transformations we refer to the course of
J.-C.~Yoccoz \cite{yocc-2005} (in particular Corollary~III.5 and the
proof of Proposition~IV.10).

For integers $d\geq 1$, $1\leq i\leq d$ and $0\leq j\leq d$ we consider
the matrices $A_d^i\in M_{d\times (d+1)}(\Z_{\geq 0})$ and $B_d^{i,j}\in M_{(d+1)\times
  d}(\Z_{\geq 0})$:
\[
A_d^i=\left(\begin{array}{c|c}
I_d&C_d^i
\end{array}\right)
\text{ et }
B_d^{i,j}=\left(\begin{array}{c|c}\makebox{}I_j&0\\\hline \multicolumn{2}{c}{L_d^i}\\\hline 0&I_{d-j}\end{array}
\right)
\]
where $C_d^i$ is the column vector of height $d$ with $0$ coefficients
except the $i$-th coefficient which is $1$ and $L_d^i=^tC_d^i$ is the
line vector of length $d$ with $0$ coefficients except the $i$-th
coefficient which is $1$.

Let $(d_n)$ be a sequence of positive integers and $(M_n)$ be a
sequence of $d_n\times d_{n+1}$ integer non-negative matrices of the
form $A_{d_n}^{i_n}$ or $B_{d_n}^{i_n,j_n}$. For such a sequence we denote
by
\[D=\liminf d_n.\]
The \textbf{positive cone} of the sequence $(M_n)_{n\in\N}$ is the set
\[
C=\{(v_n)_{n\in\N}\ |\ \forall n\in\N,\ v_n\in \R_{\geq 0}^{d_n},\text{ and }
v_n=M_nv_{n+1}\}.
\]

\begin{lem}\label{lem:matrices}
  The positive cone $C$ has projective dimension at most $D-1=\liminf d_n-1$.
\end{lem}
\begin{proof}
By linearity, the dimension of $C$ is bounded above by
$\liminf d_n$ and thus the projective dimension is bounded above by
$\liminf d_n-1$.
\end{proof}

We now state a Lemma that will be used to get a unique
ergodicity criterion in Section~\ref{sec:ue}.

\begin{lem}\label{lem:matrix-ue}
 Assume that there exists $L\geq 1$ and infinitely many $n$ such that 
\begin{enumerate}
\item $d_n=d_{n+L}=D$,
\item the square matrix $M_{[n,n+L-1]}=M_n\cdots M_{n+L-1}$ has
  strictly positive entries.
\end{enumerate}
  Then, the projective positive cone defined by the sequence $(M_n)$
  contains excatly one point.
\end{lem}
\begin{proof}
  The matrices $M_{[n,n+L-1]}$ are non-negative integer matrices.
The entries of $M_{[n,n+L-1]}$ are
  bounded by some constant depending only on $L$. The matrices
  $M_{[n,n+L-1]}$ uniformly contract the Hilbert distance in the
  positive cone. Thus $C$ has zero diameter.
\end{proof}

\section{Levitt case}\label{sec:levitt-case}

\subsection{Rips induction}\label{sec:rips-induction}

Let $S=(F,A)$ be a system of isometries with graph $\Gamma$.  Let $F'$
be the set of elements of $F$ which belongs to the domains of at least
two partial isometries in $A^{\pm 1}$. The set $F'$ is the union of
the intersections of the domains of all possible pairs of distinct
elements in $A^{\pm 1}$:
\[
F'=\{P\ |\ \exists a\neq b\in A^{\pm 1},\ P\in
\dom(a)\cap\dom(b)\}=\cup_{a\neq b\in A^{\pm 1}}\dom(a)\cap\dom(b).
\]
Thus $F'$ is also a compact forest.  Let $A'$ be the set of all
possible non-empty restrictions of elements of $A$ to pairs of
connected components of $F'$.  The system of isometries $S'=(F',A')$
is obtained from $S$ by \textbf{Rips induction}.

Let $\Gamma'$ be the graph of $S'$ and let $\tau:\Gamma'\to\Gamma$ map
a vertex $K'$ of $\Gamma'$ (which is a connected component of $F'$) to
the connected component $\tau(K')=K$ of $F$ that contains $K'$.
Similarly, $\tau$ maps the edge $a'\in A'^{\pm 1}$ to the edge $a\in
A^{\pm 1}$ of which $a'$ is a restriction.

The system of isometries $S=(F,A)$ is \textbf{reduced}~\cite{ch-a} if
\begin{enumerate}
\item the graph $\Gamma$ is connected,
\item it has independent generators,
\item for each $P\in F$ there exists at least one infinite reduced
  admissible path $X$ such that $\CQ(X)=P$ and,
\item for each partial isometry $a\in A^{\pm 1}$, and each extremal
  point $P$ in $\dom(a)$, $P$ is in $F'$.
\end{enumerate}
We remark that if the system of isometries $S$ is reduced then the
graph $\Gamma$ does not have vertices of valence $1$.

We summarize our previous work in the following Proposition:

\begin{prop}[{\cite[Propostions~3.12, 3.13 and 5.6]{ch-a}}]\label{prop:Rips-tau}
  Let $S$ be a reduced system of isometries and $S'$ be the system of
  isometries obtained by Rips induction. Then $S'$ is reduced and the
  map $\tau:\Gamma'\to\Gamma$ is a homotopy equivalence. Moreover, the
  map $\tau$ induces a one-to-one correspondance between bi-infinite
  admissible paths in $S$ and $S'$: For any bi-infinite admissible
  path $Z'\in L(S')$, the bi-infinite path $\tau(Z')$ is reduced
  and admissible and, for any bi-infinite admissible path $Z$ in
  $\Gamma$ there exists a unique bi-infinite admissible path $Z'\in
  L(S')$ such that $\tau(Z')=Z$. By abuse of notations we write
\[
L(S)=L(S').\qedhere
\]
\end{prop}\qed

We now proceed to analyse cylinders of the lamination.

\begin{prop}\label{prop:Rips-partition}
  Let $S=(F,A)$ be a reduced system of isometries with graph
  $\Gamma$. Let $S'$ be obtained by Rips induction from $S$. Let
  $\Gamma'$ be the graph of $S'$ and $\tau:\Gamma'\to\Gamma$ be the
  graph map.

Then, for each edge $e$ of $\Gamma$
\[
C(\Gamma,e)=\biguplus_{e'}C(\Gamma',e').
\]
where the disjoint union is taken over all edges $e'$ of $\Gamma'$
such that $\tau(e')=e$.
More generally, for any finite admissible path $w$ in $\Gamma$
\[
C(\Gamma,w)=\biguplus_{w'}C(\Gamma',w').
\]
where the disjoint union is taken over all finite reduced paths $w'$ of $\Gamma'$
such that $\tau(w')=w$.

\end{prop}
\begin{proof}
  First recall that the equalities in the Proposition are understood
  through the identifaction of $L(S)$ and $L(S')$ via the map $\tau$.

  From Proposition~\ref{prop:Rips-tau}, for each bi-infinite
  admissible path $Z\in C(\Gamma,w)$ there exists a unique bi-infinite
  admissible path $Z'\in L(S')$ such that $\tau(Z')=Z$. Let $w'$ be
  the prefix of $Z'$ of length $|w|$, then $\tau(w')=w$. This
  proves that $Z'$ is in $C(\Gamma',w')$ and that $w'$ is
  unique. Conversely for any path $w'$ of $\Gamma'$ such that
  $\tau(w')=w$ and any bi-infinite reduced path $Z'\in C(\Gamma',w')$,
  $\tau(Z')$ is a bi-infinite reduced admissible path in $C(\Gamma,w)$.
\end{proof}

\subsection{Analysis of the lamination}\label{sec:rips-lamination}

Let $T$ be an $\R$-tree with a free, minimal, action of the free group
$\FN$ with dense orbits. Let $A$ be a basis for $\FN$, and let
$S_0=S_A=(K_A,A)$ be the corresponding system of isometries. Recall
that $S_0$ is reduced~\cite[Proposition~5.6]{ch-a}. We perform
inductively Rips induction to get a sequence $S_n=(F_n,\Gamma_n)$ of
reduced systems of isometries, together with maps
$\tau_n:\Gamma_{n+1}\to\Gamma_n$. By Propositions~\ref{prop:lam-repr}
and \ref{prop:Rips-tau}, for each $n$ we have that
$L(T)=L(S_0)=L(S_n)$.

For a vertex $v_n$ in $\Gamma_n$ we denote by $C(\Gamma_n,v_n)$ the
set of bi-infinite admissible paths that goes through $v_n$ at
index $0$. From Proposition~\ref{prop:Rips-tau} we get that
$\tau_{n-1}(C(\Gamma_n,v_n))\subseteq
C(\Gamma_{n-1},\tau_{n-1}(v_n))$. Again by abuse of notations we
simply write $C(\Gamma_n,v_n)\subseteq
C(\Gamma_{n-1},\tau_{n-1}(v_n))$. Recall also that $v_n$ is a
connected component of the forest $F_n$ and that for each $Z\in
C(\Gamma_n,v_n)$, $\CQ(Z)$ is a point in $v_n$.
 
Let now $(v_n)_{n\in\N}$ be a sequence such that for each $n\in\N$,
$v_n$ is a vertex in $\Gamma_n$ and $\tau_n(v_{n+1})=v_n$. The
connected components $v_n$ of $F_n$ are nested and we denote by
$v_\infty\subset K_A$ their intersection.  We proved~\cite{ch-a} that
the tree $T$ is of Levitt type if and only if for each such sequence
$v_\infty$ consists of a single point.  The cylinders
$C(\Gamma_{n},v_{n})$ are also nested and their intersection is the
set of bi-infinite admissible paths $Z$ in $\Gamma_0$ such that
$\CQ(Z)\in v_\infty$. As the action on $T$ is free, the map $\CQ$ is
finite-to-one~\cite[Corollary~5.4]{ch-a} and thus we proved

\begin{prop}\label{prop:rips-finite-intersections}
  Let $T$ be an $\R$-tree in $\partial\cvn$ with dense orbits and of
  Levitt type.  Let $A$ be a basis of $\FN$ and $(S_n)$ be the
  sequence of systems of isometries obtained from $S_0=S_A=(K_A,A)$ by
  Rips induction. Let $(v_n)$ be sequence of vertices of the graphs
  $\Gamma_n$ such that $\tau_n(v_{n+1})=v_n$. Then, the nested
  intersection of the compact subtrees $v_n$ of $K_A$ is a singleton
  and the nested intersection of the cylinders $C(\Gamma_n,v_n)$ is
  finite. \qed
\end{prop}

\subsection{Rips induction and currents}

Let $T$ be an $\R$-tree with a free action of $\FN$ by isometries with
dense orbits. Let $A$ be a basis of $\FN$ and $S_0=S_A=(K_A,A)$ be the
associated system of isometries. Let $(S_n)_{n\in\N}$ be the sequence of
systems of isometries obtained from $S_0$ by Rips induction. By
Proposition~\ref{prop:Rips-tau} for each $n\in\N$, we have
$L(S_n)=L(S)$. Let $\Gamma_n$ be the graph of $S_n$ and for each edge
$e$ recall that $C(\Gamma_n,e)$ is the set of bi-infinite admissible
 paths in $\Gamma_n$ that goes through $e$ at index
$0$.

For a current $\mu$ supported on the dual lamination
$L(T)=L(S_0)=L(S_n)$, for each $n\in\N$, for each edge $e$ of
$\Gamma_n$ we consider the finite number
$\mu(e)=\mu(C(\Gamma_n,e))$. If $e$ and $e'$ are two consecutive edges
of $\Gamma_n$ separated by a valence $2$ vertex, any path through $e$
goes through $e'$, and thus $\mu(e)=\mu(e')$. A \textbf{generalized
  edge} $\hat e$ of $\Gamma_n$ is a maximal reduced path such that all
inner vertices have valence exactly $2$. Recall that the maps $\tau_n$
are homotopy equivalences, that $\Gamma_0$ is the rose with $N$ petals
and that the graphs $\Gamma_n$ are connected and do not have vertices
of valence $1$. The graph $\Gamma_n$ has at most $3N-3$ generalized
edges. We denote by $\GE(\Gamma_n)$ the set of generalized edges of
$\Gamma_n$. We pick-up arbitrarily an edge $e$ in each generalized
edge $\hat e$ in $\GE(\Gamma_n)$. As the system of isometries $S_n$
is reduced, there is a bi-infinite admissible path $Z$ going through
each edge $e$ of $\Gamma_n$ and thus through each generalized edge
$\hat e$. By Propostion~\ref{prop:Rips-tau},
$\tau_0\circ\cdots\circ\tau_{n-1}(Z)$ is a bi-infinite reduced path in
$\Gamma_0$ and in particular its finite subpath
$\tau_0\circ\cdots\circ\tau_{n-1}(\hat e)$ is reduced.

The \textbf{incidence matrix} $M_n$ of $\tau_n$ is the non-negative
matrix such that for each pair of generalized edges $\hat e\in
\GE(\Gamma_n)$ and $\hat e'\in \GE(\Gamma_{n+1})$, the coefficient
$M_n(\hat e,\hat e')$ is the number of occurences of $e$ in the
reduced finite path $\tau_n(\hat e')$. We remark that this matrix
depends on the choice of the edge $e$ in the generalized edge $\hat
e$. 

To the current $\mu$ carried by $L(S)$, for each $n\in\N$, we
associate the non-negative vector $\mu_n=(\mu(\hat e))_{\hat
  e\in\GE(\Gamma_n)}$.

\begin{prop}\label{prop:rips-matrix}
  Let $T$ be an $\R$-tree with a free action of $\FN$ by isometries
  with dense orbits. Let $A$ be a basis of $\FN$ and $S_0=S_A=(K_A,A)$
  be the associated system of isometries. Let $(S_n)_{n\in\N}$ be the
  sequence of systems of isometries obtained from $S_0$ by Rips
  induction. Let $\Gamma_n$ be the graph of $S_n$ and $GE(\Gamma_n)$
  its set of generalized edge and, let $M_n$ be the incidence matrix
  of $\tau_n:\Gamma_{n+1}\to\Gamma_{n}$.

  Then, for each current $\mu$ carried by $L(T)$
\[
\mu_n=M_n\mu_{n+1}
\]
where $\mu_n=(\mu(C(\Gamma_n,\hat e)))_{\hat e\in\GE(\Gamma_n)}$ is
the non-negative vector associated to $\mu$ at step $n$.
\end{prop}
\begin{proof}
  From Proposition~\ref{prop:Rips-partition}, we have
  \[
  \mu(\hat e)=\mu(e)=\sum_{e'}\mu(e')=\sum_{e'}\mu(\hat e'),
  \]
  where the sum is taken over all edges $e'$ such that
  $\tau(e')=e$. Grouping together the generalized edges of $\Gamma'$,
  proves the Proposition.
\end{proof}

We now use the fact that the tree is of Levitt type:

\begin{prop}\label{prop:rips-exhaustion}
  With the above notations, let $(v_n)_{n\in\N}$ be a sequence such
  that for each $n$, $v_n$ is a vertex of $\Gamma_n$ and
  $\tau_n(v_{n+1})=v_n$. Then the sequence $(\mu(C(\Gamma_n,v_n)))_{n\in\N}$ is
  non-increasing and converges to $0$.
\end{prop}
\begin{proof}
  Recall from Section~\ref{sec:rips-lamination} that the cylinder
  $C(\Gamma_n,v_n)$ is the set of bi-infinite reduced admissible paths
  that goes through $v_n$ at index $0$. The cylinders
  $C(\Gamma_n,v_n)$ are nested and from
  Proposition~\ref{prop:rips-finite-intersections} their intersection
  $C(v_\infty)$ is finite. By Proposition~\ref{prop:periodic-leaves}
  the current $\mu$ has no atomes and $\mu(C(v_\infty))$ is null,
  which proves the Proposition.
\end{proof}

\subsection{The simplex of currents}

We are now ready to prove Theorem~\ref{thm:main} in the case where the
Rips induction completly decomposes the tree.

\begin{thm}\label{thm:main-levitt}
  Let $T$ be an $\R$-tree with a minimal free action of $\FN$ by
  isometries with dense orbits. Assume that $T$ is a tree of Levitt
  type.

  Then the simplex $\Curr(T)$ of currents carried by the dual
  lamination $L(T)$ has dimension at most $3N-4$ (and projective
  dimension at most $3N-5$).
\end{thm}
\begin{proof}
  Let $A$ be a basis of $\FN$ and $S_0=S_A=(K_A,A)$ be the associated
  system of isometries. We perform the Rips induction to get a
  sequence of systems of isometries $S_n$ together with maps
  $\tau_n:\Gamma_{n+1}\to\Gamma_n$.

  For each $\mu\in\Curr(T)$, we consider the sequence of vectors
  $(\mu_n)_{n\in\N}$, where $\mu_n=(\mu_n(\hat e))_{\hat e\in 
    \GE(\Gamma_n)}$. We denote by $C=\{\ (\mu_n)_{n\in\N}\ |\
  \mu\in\Curr(T)\ \}$ the positive cone spanned by these vectors.  For
  each $n$, let $M_n$ be the incidence matrix of the map
  $\tau_n$. From Proposition~\ref{prop:rips-matrix}, $C$ is the cone
  associated to the sequence of matrices $(M_n)$ as in
  Section~\ref{sec:matrices}.

  The incidence matrices $M_n$ are products of matrices of the form
  $A_d^i$ and $B_d^{i,j}$ (see Section~\ref{sec:matrices}). Indeed,
  the matrix $A_d^i$ is that of splitting the $i$-th edge of
  $\Gamma_n$ into two edges and the matrix $B_d^{i,j}$ is that of
  replacing the reduced path made of the the $i$-th and $j$-th edges of
  $\Gamma_{n+1}$ when they form a generalized edge. Edges splitted
  in the Rips induction are incident to branch vertices of $\Gamma_n$
  and the number of resulting splitted edges is at most the valence of
  that branch vertex.  Moreover, these splittings creates as many
  vertices as the number of new edges. Thus, both the numbers of
  matrices $A_d^i$ and $B_d^j$ in each $M_n$ is bounded by
  $2N-2$. Because $T$ is of Levitt type, the Rips machine goes for
  ever and the number of edges of $\Gamma_n$ goes to infinity. 

  From Lemma~\ref{lem:matrices}, the projective dimension of
  $C$ is at most $D-2$ where $D$ is the inferior limit of the number
  of generalized edges in $\Gamma_n$. As $\Gamma_n$ is homotopic to
  the rose with $N$-petals, the number of generalized edges is bounded
  above by $3N-3$. The Rips induction applied to a graph $\Gamma_n$
  with $3N-3$ generalized edges eventually produces a graph with
  strictly less than $3N-3$ generalized edges. Thus we get that $D<
  3N-3$ and the projective dimension of $C$ is at most $3N-5$.

  Recall from Section~\ref{sec:currents} that for a finite reduced
  word $w$ in $A^{\pm 1}$ the cylinder $C_A(w)\subseteq\partial^2\FN$
  is the clopen set of bi-infinite reduced words in $A^{\pm 1}$ that
  reads $w$ at index $0$. The set of translates of all such cylinders
  is a sub-basis of open sets of the shift of bi-infinite reduced
  words in $A^{\pm 1}$. Thus a current is completely determined by the
  measures of these cylinders.

  Recall that a a generalized edge $\hat e$ in $\Gamma_n$ is mapped by
  $\tau_0\circ\cdots\tau_{n-1}$ to a finite reduced admissible path in
  $\Gamma_0$ and that $\Gamma_0$ is the rose with $N$ petals labelled
  by $A$.

  We denote by $\langle \hat e|w\rangle $ the number of
  occurences of the word $w$ in the reduced word
  $\tau_0\circ\cdots\tau_{n-1}(\hat e)$. We claim that for a current $\mu\in\Curr(T)$,
\[
\mu(w)=\lim_{n\to+\infty}\sum_{\hat e\in\GE(\Gamma_n)}\langle \hat e|w\rangle \mu(\hat e).
\]
This formula proves that $\mu$ is completely determined by the image
of $\mu$ in $C$ and thus proves the Proposition.  We now prove the above formula.

Using Proposition~\ref{prop:Rips-tau}, we get
\[
C(\Gamma_0,w)=\biguplus_{w'} C(\Gamma_n,w')
\]
where the disjoint union is taken over all reduced paths $w'$ in
$\Gamma_n$ with label $w$. Passing to the current we get
\[
\mu(w)=\sum_{w'}\mu(C(\Gamma_n,w')).
\]
If $w'$ is a subpath of a generalized edge $\hat e$ of $\Gamma_n$,
then all bi-infinite reduced paths through $w'$ goes through $\hat e$
and thus $\mu(C(\Gamma_n,w'))=\mu(\hat e)$. Taking into account only
these occurences of the label $w$ inside generalized edges, we get the
inequality
\[
\mu(w)\geq\sum_{\hat e\in\GE(\Gamma_n)}\langle \hat e|w\rangle \mu(\hat e).
\]
The difference between the two terms above is exactly the measure of
the cylinders of the occurences of the label $w$ that are not inside a
generalized edge. Those occurences cross a point of valence at least
$3$. For such a vertex $v$ in $\Gamma_n$ of valence at least $3$, we
denote by $\langle v|w\rangle $ the number of reduced paths in $\Gamma_n$ with label
$w$ that pass through $v$. We remark that $\langle v|w\rangle $ is bounded above by
$(|w|-1)(2N-1)^{|w|}$, where $|w|-1$ is the number of vertices
inside a path of length $|w|$ and $2N$ is the maximal valence of a
vertex in $\Gamma_n$. (Note that this bound is verly loose but
sufficient for our proof). We also use the notation
$\mu(v)=\mu(C(\Gamma_n,v))$. The difference of the two
terms in the above inequality is bounded by
\[
0\leq\mu(w)-\sum_{\hat e\in\GE(\Gamma_n)}\langle \hat e|w\rangle \mu(e)\leq
\sum_v\langle v,w\rangle \mu(v)\leq (|w|-1)(2N-1)^{|w|}\sum_v\mu(v)
\]
where $v$ ranges over all branch-points  of
$\Gamma_n$. Now $\Gamma_n$ is homotopy equivalent to $\Gamma_0$
and does not have vertices of valence $1$, thus the number of
branch-points in $\Gamma_n$ is bounded above by $2N-2$. We get
\[
0\leq\mu(w)-\sum_{\hat e\in\GE(\Gamma_n)}\langle\hat e|w\rangle\mu(e)\leq
(|w|-1)(2N-1)^{|w|}(2N-2)\max_v\mu(v)
\]
By Proposition~\ref{prop:rips-exhaustion}, the last factor
goes to $0$ when $n$ goes to infinity which concludes the proof.
\end{proof}

\section{Splitting induction}\label{sec:non-levitt}

There are trees in the boundary of Outer space where the Rips
induction is unuseful: the trees of surface type~\cite{ch-a}. For
these trees (and in general) we define another kind of induction,
which we call splitting induction. It is a little more involved to use
it for analysing the dual lamination as we are going to miss finitely
many leaves. But this gives us the same results for currents when they
have no atomes.

\subsection{Splitting}\label{sec:splitting}

This Section recalls the definitions and results of our previous work
with P.~Reynolds~\cite[Sections~4.2 and 4.3]{chr}.

Let $S=(F,A)$ be a reduced system of isometries.  A point $P\in
F$ is in the \textbf{interior} of a tree $K\subseteq F$ if $P$ is in
the interior of a segment contained in $K$. A point $P$ is
\textbf{extremal} in a tree $K$ if it is not in the interior.

Let $x$ be an interior point of a connected component $K_x$ of
$F$. Let $\pi_0(K_x\ssm\{x\})=L\biguplus R$ be a partition of the set
of directions at $x$ in two non-empty subsets.  Then $(x,L,R)$ is a
\textbf{splitting partition} of the system of isometries $S=(F,A)$ if
\begin{enumerate}
\item there exists exactly one partial isometry $a_0\in A^{\pm 1}$ defined
at $x$ whose domain $\dom(a_0)$ meets both sets of directions $L$ and
$R$ and,
\item for each direction $d\in L\cup R$ at $x$, there is at least
  another partial isometry $a\in A^{\pm 1}\ssm\{a_0\}$ defined at $d$:
  $x\in\dom(a)$ and $\dom(a)\cap d\neq\emptyset$.
\end{enumerate}

For a splitting partition $(x,L,R)$, let $F'$ be the forest 
obtained by splitting $K_x$ into two disjoint compact trees $\bar L$ and
$\bar R$:
\[
F'=(F\ssm K_x) \biguplus (L\cup\{x\})\biguplus (R\cup\{x\}).
\]
We insist that there are two distinct copies of $x$ in $F'$, which we
denote by $x_L$ and $x_R$.  

We also split the partial isometries in $A$. The partial isometry
$a_0$ is replaced by its two restrictions: $a_0^L=
{_{L\cup\{x_L\}\rceil}a_0}$ and $a_0^R=
{_{R\cup\{x_R\}\rceil}{a_0}}$. Let $a\in A^{\pm 1}$ be a partial
isometry distinct from $a_0$ and $a_0\inv$. If $\dom(a)$ is not
reduced to a point, we let $a'$ be the closure of the restriction (at
the source and the target) of $a$ to $F\ssm\{x\}$. For instance, if
$x$ is in the domain of $a$ then by definition of a splitting point,
$x$ is extremal in $\dom(a)$ and $\dom(a)$ meets either $L$ or
$R$, say $L$. In this case $\dom(a')=\{x_L\}\cup(\dom(a)\cap L)$.
Finally, if $\dom(a)$ is a singleton $\{y\}$ and $y.a=y'$ we define
$a'$ arbitrarily by letting
\[
\begin{cases}
y.a'=y'\text{ if }x\not\in\{y,y'\}\\
x_L.a'=x_L\text{ if }y=y'=x\\
x_L.a'=y'\text{ if }y=x\text{ and }y'\neq x\\
y.a'=x_L\text{ if }y\neq x\text{ and }y'=x.
\end{cases}
\]

Considering the graphs $\Gamma$ and $\Gamma'$ of the systems of
isometries $S$ and $S'$, as above there is a natural map
$\tau:\Gamma'\to\Gamma$ which is one-to-one, except that it maps both
vertices $\bar L$ and $\bar R$ of $\Gamma'$ to the vertex $K_x$ of
$\Gamma$ and, except that it maps both edges $a_0^L$ and $a_0^R$ of
$\Gamma'$ to the edge $a_0$ of $\Gamma$. Indeed the map $\tau$ is a
folding.

The splitting induction affects the admissible paths that contain the
subpaths ${a_0^R}\inv a_0^L$ or ${a_0^L}\inv a_0^R$. We thus restrict
to a sublamination. A finite admissible path $u$ in $\Gamma$ is
\textbf{regular} if $\dom(u)$ contains strictly more than one
point. The \textbf{regular lamination} $L'(S)$ of a system of
isometries $S$ is the set of bi-infinite admissible paths $Z$ in
$\Gamma$ such that all finite subpaths $u$ of $Z$ are regular. This
regular lamination is also the derived set $L'(S)$ of non-isolated
leaves in $L(S)$:

\begin{prop}[\cite{chr}]\label{prop:regular-derived}
  Let $T$ be an $\R$-tree with a free minimal action of $\FN$ by
  isometries with dense orbits. Let $A$ be a basis of $\FN$ and
  $S_A=(K_A,A)$ be the associated system of isometries. Recall that
  the dual lamination $L(T)$ of $T$ is equal to the admissible
  lamination of $S_A$.

  Then, the regular lamination $L'(S_A)$ is the derived space of
  $L(S_A)$. \qed   
\end{prop}

We summarize our previous work with P.~Reynolds in the following
Proposition:

\begin{prop}[{\cite[Lemma~4.8 and 4.9]{chr}}]\label{prop:split-tau}
  Let $S$ be a reduced system of isometries and let $S'$ be a system
  of isometries obtained by splitting. Then $S'$ is reduced and the map
  $\tau:\Gamma'\to\Gamma$ is a homotopy equivalence.

  For any bi-infinite reduced regular path $Z'\in L(S')$ in $\Gamma'$,
  the bi-infinite path $\tau(Z')$ is reduced, admissible and
  regular. For any bi-infinite reduced regular path $Z\in L(S)$ in
  $\Gamma$, there exists a unique bi-infinite reduced regular path
  $Z'\in L(S')$ such that $\tau(Z')=Z$.

  Through the map $\tau$, we identify the regular laminations:
\[
L'(S)=L'(S').
\]
\end{prop}

We now write the effect of splitting on the cylinders of the
lamination of $S$.  As in Section~\ref{sec:levitt-case}, for a finite
reduced path $u$ of $\Gamma$ ($u$ can be a vertex or an edge), the
cylinder $C(\Gamma,u)$ is the set of bi-infinite admissible paths that
goes through $u$ at index $0$. But here we rather consider the regular
cylinder $C'(\Gamma,u)$: the set of bi-infinite reduced admissible
regular paths that goes through $u$ at index $0$. Exactly as in
Proposition~\ref{prop:Rips-partition} we have:

\begin{prop}\label{prop:split-partition}
  Let $S=(F,A)$ be a reduced system of isometries with graph $\Gamma$.
  Let $S'$ be obtained by splitting induction from $S$. Let $\Gamma'$
  be the graph of $S'$ and $\tau:\Gamma'\to\Gamma$ be the graph map.
  For each edge $e$ of $\Gamma$:
\[
C'(\Gamma,e)=\biguplus_{e'}C'(\Gamma',e')
\]
where the disjoint union is taken over all edges $e'$ of $\Gamma'$ such that
$\tau(e')=e$. More generally, for every finite regular path $u$ in
$\Gamma$
\[
C'(\Gamma,u)=\biguplus_{u'}C'(\Gamma',u')
\]
where the disjoint union is taken over all finite reduced paths $u'$
of $\Gamma'$ such that $\tau(u')=u$.\qed
\end{prop}

\subsection{Existence of splittings}\label{sec:indec-surface}

To use the splitting induction to analyse laminations and currents we
first need to prove that we can perform it: there exists splitting
partitions. From our previous work with P.~Reynolds~\cite{chr} we know
that this is the case, at least when we cannot perform Rips
induction. A system of isometries $S=(F,A)$ is of \textbf{surface
  type} if any point $x$ of $F$ belongs to the domains of at least two
partial isometries $a\neq b\in A^{\pm 1}$ (equivalently Rips induction
does nothing to $S$).

\begin{prop}[{\cite[Proposition~4.11]{chr}}]\label{prop:split-exist}
  Let $S=(F,A)$ be a reduced system of isometries of surface type,
  then there exists a splitting partition for $S$. \qed
\end{prop}

\section{Unfolding}\label{sec:indecomposable}

\subsection{Unfolding induction and currents}

Let $T$ be an $\R$-tree with a free minimal action of $\FN$ by
isometries with dense orbits. Let $L(T)$ be the dual lamination of
$T$. By Proposition~\ref{prop:periodic-leaves}, no leaf of $L(T)$ is
periodic and, any current $\mu\in\Curr(T)$ carried by $L(T)$ has
no  atomes. By definition of the derived space we get:

\begin{prop}\label{sec:current-derived}
  Let $T$ be an $\R$-tree with a free minimal action of $\FN$ by
  isometries with dense orbits. Then every current $\mu\in\Curr(T)$
  carried by the dual lamination $L(T)$ is carried by the regular
  lamination $L'(T)$. \qed
\end{prop}

Let $A$ be a basis of $\FN$ and $S_A=(K_A,A)$ the associated system of
isometries.  Let $S_n=(F_n,A_n)$ be a sequence of systems of
isometries obtained from $S_0=S_A$ by \textbf{unfolding induction}: at
each step $n$ either Rips or splitting induction is performed. Let
$\tau_n:\Gamma_{n+1}\to\Gamma_{n}$ be the map between the graphs. As
$S_0$ is reduced, $\tau_n$ is a homotopy equivalence, $\Gamma_n$ is
connected and does not have vertices of valence $1$. The graph
$\Gamma_n$ has at most $3N-3$ generalized edges in $\GE(\Gamma_n)$.

For each generalized edge $\hat e'\in\GE(\Gamma_{n+1})$ and each edge
$e$ in $\Gamma_n$ we consider $\langle \hat e',e\rangle$ the number of
occurences of $e$ (without taking into account orientation) in the
finite path $\tau(\hat e)$. The \textbf{incidence matrix} $M_n$ of the
map $\tau_n$ has entry $\langle \hat e',e\rangle$ for each pair $(\hat
e,\hat e')\in\GE(\Gamma_{n})\times\GE(\Gamma_{n+1})$. Again, we remark
that the incidence matrix depends on the choice of the edge $e$ in the
generalized edge $\hat e$.

For a generalized edge $\hat e$ of $\Gamma_n$ we
consider the non-negative number $\mu(C'(\Gamma_n,\hat e))$, where
$C'(\Gamma_n,\hat e)$ is the cylinder of regular bi-infinite admissible
paths in $\Gamma_n$ that passes through $\hat e$ at index $0$. We consider
the non-negative vector $\mu_n=(\mu(C'(\Gamma_n,\hat e)))_{\hat
  e\in\GE(\Gamma_n)}$.

\begin{prop}\label{prop:split-matrix}
  Let $T$ be an $\R$-tree with a free action of $\FN$ by isometries
  with dense orbits. Let $A$ be a basis of $\FN$ and $S_0=S_A=(K_A,A)$
  be the associated system of isometries. Let $(S_n)_{n\in\N}$ be a
  sequence of systems of isometries obtained from $S_0$ by unfolding
  induction. Let $\Gamma_n$ be the graph of $S_n$ and $\GE(\Gamma_n)$
  its set of generalized edge and, let $M_n$ be the incidence matrix
  of $\tau_n:\Gamma_{n+1}\to\Gamma_{n}$.

  Then, for each current $\mu$ carried by $L(T)$
\[
\mu_n=M_n\mu_{n+1}
\]
where $\mu_n=(\mu(C'(\Gamma_n,\hat e)))_{\hat e\in\GE(\Gamma_n)}$ is
the non-negative vector associated to $\mu$ at step $n$.
\end{prop}
\begin{proof}
  The proof is the same as the proof of
  Proposition~\ref{prop:rips-matrix}, using
  Proposition~\ref{prop:split-partition} if we use splitting induction
  at step $n$.
\end{proof}

We can now prove Theorem~\ref{thm:main} when the induction completly
analyses the lamination.

\begin{thm}\label{thm:main-under-complete}
  Let $T$ be an $\R$-tree with a minimal, free, action of $\FN$ by
  isometries with dense orbits. Let $A$ be a basis of $\FN$ and let
  $S_A=(K_A,A)$ be the associated system of isometries. Let
  $S_n=(F_n,A_n)$ be a sequence of systems of isometries obtained from
  $S_0=S_A$ by unfolding induction.

  Assume that each nested intersection of connected components $v_n$
  of $F_n$ is a singleton.

  Then, the simplex of currents $\Curr(T)$ carried by the dual
  lamination $L(T)$ has dimension at most $3N-4$ (and projective
  dimension at most $3N-5$).
\end{thm}
\begin{proof}
  Again the proof is the same as the proof of
  Theorem~\ref{thm:main-levitt}, using the regular lamination instead
  of the lamination and Proposition~\ref{prop:split-matrix}.
\end{proof}

\subsection{Decomposable case}\label{sec:induction-on-n}

The hypotesis of Theorem~\ref{thm:main-under-complete} are not always
satisfied: there exist $\R$-trees with no induction sequence that
completly analyses the lamination. Those trees are decomposable in the
sense of V.~Guirardel~\cite{guir-indecomposable}. In this case
however, the induction procedure ends up with a subtree with an
action of a subgroup of $\FN$ with rank strictly smaller than $N$. And
this allows us to conclude the proof of Theorem~\ref{thm:main} by
induction on $N$.

Let $T$ be an $\R$-tree with a minimal free action of $\FN$ with dense
orbits. Let $A$ be a basis of $\FN$ and let $S_A=(K_A,A)$ be the
associated system of isometries. Let $S_n=(F_n,A_n)$ be an infinite
sequence of systems of isometries obtained from $S_0=S_A$ by unfolding
induction. By Proposition~\ref{prop:split-exist}, this is
always possible. Let $\tau_n:\Gamma_{n+1}\to\Gamma_n$ be the map
between the graphs. As $S_n$ is reduced, $\Gamma_n$ is connected, does
not have vertices of valence $1$ and $\tau_n$ is a homotopy
equivalence. The number of edges of $\Gamma_n$ is strictly increasing
with $n$. The inverse limit of $(\Gamma_n,\tau_n)$ is an infinite
graph $\hat\Gamma$. Vertices of $\hat\Gamma$ are sequences
$(v_n)_{n\in\N}$ such that $v_n$ is a vertex of $\Gamma_n$ and
$\tau_n(v_{n+1})=v_n$, thus $(v_n)_{n\in\N}$ is a sequence of nested
compact subtrees of $K_A$. Edges of $\hat\Gamma$ are sequences
$(a_n)_{n\in\N}$ such that $a_n$ is an edge of $\Gamma_n$ and
$\tau_n(a_{n+1})=a_n$, thus $a_{n+1}$ is a restriction of the partial
isometry $a_n$.

By our previous work \cite{ch-a}, the index of $\hat\Gamma$ is finite
and thus it contains a finite \textbf{core graph} $\Gamma_\infty$:
the union of all reduced loops in $\hat\Gamma$. We remark that
$\Gamma_\infty$ can be empty and that it can fail to be connected. 

\underline{First case}: $\Gamma_\infty$ is empty. Equivalently, each
connected component of $\hat\Gamma$ is a tree. Moreover, as the index
of $\hat\Gamma$ is finite, each connected component of $\hat\Gamma$
has a finite number of ends. Let $(v_n)$ be a vertex of $\hat\Gamma$,
this is a sequence of vertices of $\Gamma_n$ which are nested
connected components of the forest $F_n$. As before, we denote by
$v_\infty$ the nested intersection of $(v_n)$. There are finitely many
infinite reduced paths in $\hat\Gamma$ starting from $v_\infty$. For
each $x\in v_\infty$ there exists an infinite path in $\hat\Gamma$
starting at $(v_n)$ which reads an admissible reduced regular word $X$
such that $\CQ(X)=x$. We get that $v_\infty$ is finite and thus a
singleton. Theorem~\ref{thm:main-under-complete} applies in this case.

\underline{Second case}: $\Gamma_\infty$ is non-empty. As
$\Gamma_\infty$ is finite, there exists $s\in\N$ such that for all
$n\geq s$, $\Gamma_\infty$ is a subgraph of $\Gamma_n$ and the map
$\tau_n$ restricts to the identity on $\Gamma_\infty$. We denote by
$\Gamma_\infty^1,\ldots,\Gamma_\infty^r$ the connected components of
$\Gamma_\infty$ in this case $r\geq 1$. For each $i=1,\ldots,r$, let
$F_\infty^i$ be the compact forest whose connected components are the
nested intersections of the vertices of $\Gamma_\infty^i$, and let
$A_\infty^i$ be the partial isometries of $F_\infty^i$ which are the
nested intersections of edges of $\Gamma_\infty^i$.

\begin{lem}\label{lem:}
  The system of isometries $S_\infty^i=(F_\infty^i,A_\infty^i)$ has
  graph $\Gamma_\infty^i$ and, is reduced.

  Let $N_i$ be the rank of the free group
  $\pi_1(\Gamma_\infty^i)$. Let $T_i$ be the $\R$-tree associated to
  $S_\infty^i$. The action of $F_{N_i}$ on $T_i$ is free, minimal by
  isometries and with dense orbits. \qed
\end{lem}

For any finite admissible non-singular subpath $u$ in
$\Gamma_\infty^i$, we consider the regular cylinder
$C'(\Gamma_\infty^i,u)$ of bi-infinite admissible regular paths that
passes through $u$ at index $0$. The translates of such cylinders form
a basis of open sets of the regular dual lamination $L'(T_i)$. We
insist that we consider $L'(T_i)$ as a lamination with respect to
$F_{N_i}=\pi_1(\Gamma_\infty^i)$ (and not with respect to the original
$\FN$). Any current $\mu\in\Curr(T)$ carried by the dual lamination
$L(T)$ induces a current $\mu^i\in\Curr(T_i)$ by letting
\[
\mu^i(C(\Gamma_\infty^i,u))=\mu(C(\Gamma_\infty^i,u)).
\]
For each $n>s$, as the homotopy equivalence $\tau_n$ fixes the
subgraph $\Gamma_\infty$ the incidence matrix $M_n$ is reducible and
we denote by $M^{\GE}_n$ its submatrix corresponding to generalized
edges in $\GE(\Gamma_n\ssm\Gamma_\infty)$.  Let $C$ be the positive
cone of non-negative vectors $((\mu_n(\hat e))_{\hat
  e\in\GE(\Gamma_n\ssm\Gamma_\infty)})_{n\geq s}$. 

\begin{lem}\label{lem:injective}
The map
\[
\begin{array}{rcl}
\Curr(T)&\to&\Curr(T_1)\times\cdots\times\Curr(T_r)\times C\\
\mu&\mapsto&(\mu^1,\ldots,\mu^r,((\mu_n(\hat e))_{\hat e\in\GE(\Gamma_n\ssm\Gamma_\infty)})_{n\geq s})
\end{array}
\]
is injective.
\end{lem}
\begin{proof}
  Indeed for any finite reduced word $w$ in $A^{\pm 1}$, we claim that
\[
\mu(w)=\mu(C(\Gamma_0,w))=\sum_{i=1}^{i=r}\sum_{w'\text{ occurence of
  }w\text{ in }\Gamma_\infty^i} \mu^i(w')+\lim_{n\to\infty}\sum_{\hat
  e\in\GE(\Gamma_n\ssm\Gamma_\infty)}\langle w,\hat
e\rangle\mu(\hat e).
\]
As in the proof of Theorem~\ref{thm:main-levitt},
Propositions~\ref{prop:Rips-tau} and \ref{prop:split-tau} prove that
for each $n$:
\[
\mu(w)\geq\sum_{i=1}^{i=r}\sum_{w'\text{ occurence of
  }w\text{ in }\Gamma_\infty^i} \mu^i(w')+\sum_{\hat
  e\in\GE(\Gamma_n\ssm\Gamma_\infty)}\langle w,\hat
e\rangle\mu(\hat e).
\]
and that the difference is given by
\[
\Delta_n=\sum_{w'}\mu(C'(\Gamma_n,w'))
\]
where the sum is taken over all occurences of $w$ as a label of a path
$w'$ in $\Gamma_n$ that contains both vertices of $\Gamma_\infty$ and
edges of $\Gamma_n\ssm\Gamma_\infty$. By definition of
$\Gamma_\infty$, all vertices of $\Gamma_\infty$ have valence at least
$2$ and thus those occurences of $w$ passes over one of the branch
points of $\Gamma_n$. Therefore the number of such occurences is
uniformly bounded (by a constant which depends only on $N$). Hence,
there are finitely many occurences of $w$ in the inverse limit
$\hat\Gamma$ which contains both vertices of $\Gamma_\infty$ and edges
out of $\Gamma_\infty$.  The vertices out of the core graph of
$\hat\Gamma$ corresponds to singletons in $K_A$ and as the action of
$\FN$ on $T$ is free there are only finitely many bi-infinite
admissible regular paths $Z$ in $\hat\Gamma$ that passes through these
occurences of $w$ at index $0$.

We thus get that $\lim_{n\to\infty}\Delta_n=0$
\end{proof}

By Propostion~\ref{prop:split-matrix} and by definition of the cone
$C$, for all $(v_n)_{n\geq s}\in C$ and for $n\geq s$, we have
$v_n=M_n^\GE v_{n+1}$. Thus, the cone $C$ has finite dimension at most
the number of generalized edges in $\Gamma_n\ssm\Gamma_\infty$ minus
one. By induction on the rank of the free group, $\Curr(T_i)$ has
dimension at most $3N_i-5$. Let $\Gamma_n^\GE$ be the graph obtained
from $\Gamma_n$ by replacing generalized edges by edges (so that
$\Gamma_n^\GE$ has no vertex of valence $2$). Contracting the image of
each component $\Gamma_\infty^i$ in $\Gamma_n^\GE$ to a vertex, we get
a connected graph $G_n$. By construction $G_n$ has no vertex of
valence $1$, at most $r$ vertices of valence $2$ and
$|\GE(\Gamma_n\ssm\Gamma_\infty)|$ edges. Moreover $G_n$ has Euler
characteristic $1-N+N_1+\cdots+N_r$. Using that the number of edges of
a connected graph without vertex of valence $1$ is bounded above by
three times the Euler characteristic plus the number of vertices of
valence $2$, we get
\[
|\GE(\Gamma_n\ssm\Gamma_\infty)|\leq 3(N-N_1-\cdots -N_r)-3+r
\]
and thus
\[
(3N_1-4)+\cdots+(3N_r-4)+|\GE(\Gamma_n\ssm\Gamma_\infty)|-1\leq 3N-4-3r
\]
which concludes the proof of Theorem~\ref{thm:main} since $r\geq 1$.

\section{Unique ergodicity}\label{sec:ue}

Recall that an $\R$-tree in $\partial\CVN$ is uniquely ergodic if it
is dual to a unique projective current.  There are known examples of
uniquely ergodic trees: As the transition matrix of a train-track of
an iwip outer automorphisms is a primitive integer matrix (up to
passing to a power), Perron-Frobenius Theorem implies that the
attracting tree in $\partial\CVN$ of a non-geometric iwip outer
automorphims is uniquely ergodic -- see for instance \cite{chl1-III}.

In this Section we give a more general concrete criterion of unique
ergodicity.

Let $T$ be an $\R$-tree in $\partial\CVN$ with dense orbits. Let $A$
be a basis of $\FN$, and $(S_n)$ a sequence of systems of isometries
obtained from $S_0=S_A=(K_A,A)$ by unfolding induction. As in
Section~\ref{sec:rips-lamination} and \ref{sec:current-derived}, let
$(\Gamma_n)$ be the sequence of associated graphs and
$\tau_n:\Gamma_{n+1}\to\Gamma_n$ the homotopy equivalences. We denote
by $D$ the inferior limit of the number $d_n$ of generalized edges of
$\Gamma_n$:
\[
D=\liminf d_n.
\]

We say that \textbf{all edges are fully splitted} between steps $n$
and $n+L$ if the image of each generalized edge $\hat e$ of
$\Gamma_{n+L}$ in $\Gamma_n$ is an edge-path that covers
$\Gamma_n$. This is equivalent to saying that the incidence matrix
(for generalized edges) of
$\tau_{n}\circ\cdots\circ\tau_{n+L-1}:\Gamma_{n+L}\to\Gamma_n$ has
strictly positive entries.

\begin{thm}
  Assume that there exists $L$ and infinitely many $n$ such that the number
  of generalized edges of $\Gamma_n$ is $D$ and between steps $n$ and $n+L$ all
  edges are fully splitted.

  Then, $T$ is uniquely ergodic.
\end{thm}
\begin{proof}
  Indeed with the hypothesis above the matrices $M_{n}\cdots
  M_{n+L-1}$ are square matrices with strictly positive and bounded
  coefficients. Thus we can use Lemma~\ref{lem:matrix-ue}.
\end{proof}

%\bibliographystyle{alpha}
%\bibliography{../bibli}

\begin{thebibliography}{CHL08b}

\bibitem[BF95]{bf-stable}
Mladen Bestvina and Mark Feighn.
\newblock Stable actions of groups on real trees.
\newblock {\em Invent. Math.}, 121(2):287--321, 1995.

\bibitem[Bon88]{bona-teich-curr}
Francis Bonahon.
\newblock The geometry of {T}eichm\"uller space via geodesic currents.
\newblock {\em Invent. Math.}, 92(1):139--162, 1988.

\bibitem[Bow99]{bowd-tree}
Brian Bowditch.
\newblock {\em Treelike structures arising from continua and convergence
  groups}, volume 662.
\newblock Memoirs Amer. Math. Soc., 1999.

\bibitem[BV06]{bv-out-teich}
Martin~R. Bridson and Karen Vogtmann.
\newblock Automorphism groups of free groups, surface groups and free abelian
  groups.
\newblock In {\em Problems on mapping class groups and related topics},
  volume~74 of {\em Proc. Sympos. Pure Math.}, pages 301--316. Amer. Math.
  Soc., Providence, RI, 2006.

\bibitem[CH14]{ch-a}
Thierry Coulbois and Arnaud Hilion.
\newblock {R}ips induction : Index of the dual lamination of an {$\R$}-tree.
\newblock {\em Groups Geom. Dyn.}, 8(1):97-134, 2014.

\bibitem[CHL07]{chl2}
Thierry Coulbois, Arnaud Hilion, and Martin Lustig.
\newblock Non-unique ergodicity, observers' topology and the dual algebraic
  lamination for {$\mathbb{R}$}-trees.
\newblock {\em Illinois J. Math.}, 51(3):897--911, 2007.

\bibitem[CHL08a]{chl1-I}
Thierry Coulbois, Arnaud Hilion, and Martin Lustig.
\newblock {$\mathbb R$}-trees and laminations for free groups. {I}. {A}lgebraic
  laminations.
\newblock {\em J. Lond. Math. Soc. (2)}, 78(3):723--736, 2008.

\bibitem[CHL08b]{chl1-II}
Thierry Coulbois, Arnaud Hilion, and Martin Lustig.
\newblock {$\mathbb R$}-trees and laminations for free groups. {II}. {T}he dual
  lamination of an {$\mathbb R$}-tree.
\newblock {\em J. Lond. Math. Soc. (2)}, 78(3):737--754, 2008.

\bibitem[CHL08c]{chl1-III}
Thierry Coulbois, Arnaud Hilion, and Martin Lustig.
\newblock {$\mathbb R$}-trees and laminations for free groups. {III}.
  {C}urrents and dual {$\mathbb R$}-tree metrics.
\newblock {\em J. Lond. Math. Soc. (2)}, 78(3):755--766, 2008.

\bibitem[CHL09]{chl4}
Thierry Coulbois, Arnaud Hilion, and Martin Lustig.
\newblock {$\mathbb R$}-trees, dual laminations, and compact systems of partial
  isometries.
\newblock {\em Math. Proc. Cambridge Phil. Soc.}, 147:345--368, 2009.

\bibitem[CHR11]{chr}
Thierry Coulbois, Arnaud Hilion, and Patrick Reynolds.
\newblock Indecomposable ${F_N}$-trees and minimal laminations.
\newblock submitted, arXiv:1110.3506, 2011.

\bibitem[CL95]{cl-verysmall}
Marshall~M. Cohen and Martin Lustig.
\newblock Very small group actions on {${\bf R}$}-trees and {D}ehn twist
  automorphisms.
\newblock {\em Topology}, 34(3):575--617, 1995.

\bibitem[CV86]{cv-moduli}
Marc Culler and Karen Vogtmann.
\newblock {Moduli of graphs and automorphisms of free groups.}
\newblock {\em Invent. Math.}, 84:91--119, 1986.

\bibitem[Dur10]{durand-cant6}
Fabien Durand.
\newblock Combinatorics on {B}ratteli diagrams and dynamical systems.
\newblock In {\em Combinatorics, automata and number theory}, volume 135 of
  {\em Encyclopedia Math. Appl.}, pages 324--372. Cambridge Univ. Press,
  Cambridge, 2010.

\bibitem[Gab97]{gab-indgen}
Damien Gaboriau.
\newblock G\'en\'erateurs ind\'ependants pour les syst\`emes d'isom\'etries de
  dimension un.
\newblock {\em Ann. Inst. Fourier (Grenoble)}, 47(1):101--122, 1997.

\bibitem[GL95]{gl-rank}
Damien Gaboriau and Gilbert Levitt.
\newblock The rank of actions on {${\mathbb R}$}-trees.
\newblock {\em Ann. Sci. \'Ecole Norm. Sup. (4)}, 28(5):549--570, 1995.

\bibitem[Gui00]{guir-dynamic}
Vincent Guirardel.
\newblock Dynamics of {${\rm Out}(F\sb n)$} on the boundary of outer space.
\newblock {\em Ann. Sci. \'Ecole Norm. Sup. (4)}, 33(4):433--465, 2000.

\bibitem[Gui08]{guir-indecomposable}
Vincent Guirardel.
\newblock Actions of finitely generated groups on {$\mathbb R$}-trees.
\newblock {\em Ann. Inst. Fourier (Grenoble)}, 58(1):159--211, 2008.

\bibitem[Kap05]{kapo-frequency}
Ilya Kapovich.
\newblock The frequency space of a free group.
\newblock {\em Internat. J. Algebra Comput.}, 15(5-6):939--969, 2005.

\bibitem[Kap06]{kapo-currents}
Ilya Kapovich.
\newblock Currents on free groups.
\newblock In {\em Topological and asymptotic aspects of group theory}, volume
  394 of {\em Contemp. Math.}, pages 149--176. Amer. Math. Soc., Providence,
  RI, 2006.

\bibitem[Kea77]{keane-non-ue-iet}
Michael Keane.
\newblock Non-ergodic interval exchange transformations.
\newblock {\em Israel J. Math.}, 26(2):188--196, 1977.

\bibitem[KL07]{kl-incompatibility}
Ilya Kapovich and Martin Lustig.
\newblock The actions of {${\rm Out}(F_k)$} on the boundary of outer space and
  on the space of currents: minimal sets and equivariant incompatibility.
\newblock {\em Ergodic Theory Dynam. Systems}, 27(3):827--847, 2007.

\bibitem[KL09]{kl-geom-intersection}
Ilya Kapovich and Martin Lustig.
\newblock Geometric intersection number and analogues of the curve complex for
  free groups.
\newblock {\em Geom. Topol.}, 13(3):1805--1833, 2009.

\bibitem[KL10]{kl-intersection}
Ilya Kapovich and Martin Lustig.
\newblock Intersection form, laminations and currents on free groups.
\newblock {\em Geom. Funct. Anal.}, 19(5):1426--1467, 2010.

\bibitem[KN76]{kn}
Harvey~B. Keynes and Dan Newton.
\newblock A 'minimal', non-uniquely ergodic interval exchange transformation.
\newblock {\em Math. Z.}, 148:101--105, 1976.

\bibitem[Lev83]{levitt-these}
Gilbert Levitt.
\newblock Feuilletages des surfaces.
\newblock Thèse, Université Paris VII, 1983.

\bibitem[LL03]{ll-north-south}
Gilbert Levitt and Martin Lustig.
\newblock Irreducible automorphisms of {$F\sb n$} have north-south dynamics on
  compactified outer space.
\newblock {\em J. Inst. Math. Jussieu}, 2(1):59--72, 2003.

\bibitem[LL08]{ll-periodic}
Gilbert Levitt and Martin Lustig.
\newblock Automorphisms of free groups have asymptotically periodic dynamics.
\newblock {\em J. Reine Angew. Math.}, 619:1--36, 2008.

\bibitem[Mas82]{masur}
Howard Masur.
\newblock Interval exchange transformations and measured foliations.
\newblock {\em Ann. of Math. (2)}, 115(1):169--200, 1982.

\bibitem[Pap86]{papa-86}
Athanase Papadopoulos.
\newblock Deux remarques sur la g\'eom\'etrie symplectique de l'espace des
  feuilletages mesur\'es sur une surface.
\newblock {\em Ann. Inst. Fourier (Grenoble)}, 36(2):127--141, 1986.

\bibitem[Pau88]{paulin-top-equiv}
Fr{\'e}d{\'e}ric Paulin.
\newblock Topologie de {G}romov \'equivariante, structures hyperboliques et
  arbres r\'eels.
\newblock {\em Invent. Math.}, 94(1):53--80, 1988.

\bibitem[Sat75]{sataev}
E.~A. Sataev.
\newblock The number of invariant measures for flows on orientable surfaces.
\newblock {\em Izv. Akad. Nauk SSSR Ser. Mat.}, 39(4):860--878, 1975.

\bibitem[Yoc05]{yocc-2005}
Jean-Christophe Yoccoz.
\newblock Échanges d'intervalles.
\newblock Cours au Collège de France, 2005.

\end{thebibliography}

%\end{document}

\end{document}